\documentclass[11pt]{article}

\usepackage{enumerate,xspace}
\usepackage{amsmath,xspace,amssymb}
\usepackage[dvips]{graphics}
\usepackage{proof}
\usepackage{latexsym}  
\usepackage{euscript}  
\usepackage{bussproofs}

\input xy
\xyoption{all}
\xyoption{2cell}
\UseAllTwocells
\title{Cartesian Differential Categories Revisited}
\author{G.S.H. Cruttwell\thanks{Thanks to Rick Blute, Robin Cockett, and Pieter Hofstra for useful discussions.} 
\\ Department of Mathematics and Computer Science, \\
Mount Allison, Sackville, NB}

\newtheorem{observation}{Remark}[section]
\newtheorem{lemma}[observation]{Lemma}  
\newtheorem{theorem}[observation]{Theorem}
\newtheorem{definition}[observation]{Definition}
\newtheorem{example}[observation]{Example}

\newtheorem{proposition}[observation]{Proposition} 
\newtheorem{corollary}[observation]{Corollary} 
 
\newcommand{\x}{\times}
\newcommand{\<}{\langle}
\renewcommand{\>}{\rangle}
\newcommand{\dr}[1]{{\bf [DR.{#1}]}}
\newcommand{\cd}[1]{{\bf [CD.{#1}]}}

\newcommand{\rst}[1]{\ensuremath{\overline{#1}}}  
\newcommand{\rs}[1]{\ensuremath{\overline{#1}\,}}  

\newcommand{\T}{\ensuremath{\mathbb T}\xspace}
\newcommand{\X}{\ensuremath{\mathbb X}\xspace}
\newcommand{\Y}{\ensuremath{\mathbb Y}\xspace}

\newcommand{\Faa}{Fa\`{a}}
\newcommand{\faa}{\mbox{\textbf{\Faa}}}

\newcommand{\p}{\pi}

\newcommand\nats{\hbox{$I \kern - .38em N$}} 
\newcommand\ints{\hbox{$Z \kern - .65em Z$}} 

\makeatletter

\newdimen\w@dth

\def\setw@dth#1#2{\setbox\z@\hbox{\scriptsize $#1$}\w@dth=\wd\z@
\setbox\@ne\hbox{\scriptsize $#2$}\ifnum\w@dth<\wd\@ne \w@dth=\wd\@ne \fi
\advance\w@dth by 1.2em}

\def\t@^#1_#2{\allowbreak\def\n@one{#1}\def\n@two{#2}\mathrel
{\setw@dth{#1}{#2}
\mathop{\hbox to \w@dth{\rightarrowfill}}\limits
\ifx\n@one\empty\else ^{\box\z@}\fi
\ifx\n@two\empty\else _{\box\@ne}\fi}}
\def\t@@^#1{\@ifnextchar_ {\t@^{#1}}{\t@^{#1}_{}}}

\def\t@left^#1_#2{\def\n@one{#1}\def\n@two{#2}\mathrel{\setw@dth{#1}{#2}
\mathop{\hbox to \w@dth{\leftarrowfill}}\limits
\ifx\n@one\empty\else ^{\box\z@}\fi
\ifx\n@two\empty\else _{\box\@ne}\fi}}
\def\t@@left^#1{\@ifnextchar_ {\t@left^{#1}}{\t@left^{#1}_{}}}

\def\two@^#1_#2{\def\n@one{#1}\def\n@two{#2}\mathrel{\setw@dth{#1}{#2}
\mathop{\vcenter{\hbox to \w@dth{\rightarrowfill}\kern-1.7ex
                 \hbox to \w@dth{\rightarrowfill}}%
       }\limits
\ifx\n@one\empty\else ^{\box\z@}\fi
\ifx\n@two\empty\else _{\box\@ne}\fi}}
\def\tw@@^#1{\@ifnextchar_ {\two@^{#1}}{\two@^{#1}_{}}}

\def\tofr@^#1_#2{\def\n@one{#1}\def\n@two{#2}\mathrel{\setw@dth{#1}{#2}
\mathop{\vcenter{\hbox to \w@dth{\rightarrowfill}\kern-1.7ex
                 \hbox to \w@dth{\leftarrowfill}}%
       }\limits
\ifx\n@one\empty\else ^{\box\z@}\fi
\ifx\n@two\empty\else _{\box\@ne}\fi}}
\def\t@fr@^#1{\@ifnextchar_ {\tofr@^{#1}}{\tofr@^{#1}_{}}}

\newdimen\W@dth
\def\setW@dth#1#2{\setbox\z@\hbox{$#1$}\W@dth=\wd\z@
\setbox\@ne\hbox{$#2$}\ifnum\W@dth<\wd\@ne \W@dth=\wd\@ne \fi
\advance\W@dth by 1.2em}

\def\T@^#1_#2{\allowbreak\def\N@one{#1}\def\N@two{#2}\mathrel
{\setW@dth{#1}{#2}
\mathop{\hbox to \W@dth{\rightarrowfill}}\limits
\ifx\N@one\empty\else ^{\box\z@}\fi
\ifx\N@two\empty\else _{\box\@ne}\fi}}
\def\T@@^#1{\@ifnextchar_ {\T@^{#1}}{\T@^{#1}_{}}}

\def\T@left^#1_#2{\def\N@one{#1}\def\N@two{#2}\mathrel{\setW@dth{#1}{#2}
\mathop{\hbox to \W@dth{\leftarrowfill}}\limits
\ifx\N@one\empty\else ^{\box\z@}\fi
\ifx\N@two\empty\else _{\box\@ne}\fi}}
\def\T@@left^#1{\@ifnextchar_ {\T@left^{#1}}{\T@left^{#1}_{}}}

\def\Tofr@^#1_#2{\def\N@one{#1}\def\N@two{#2}\mathrel{\setW@dth{#1}{#2}
\mathop{\vcenter{\hbox to \W@dth{\rightarrowfill}\kern-1.7ex
                 \hbox to \W@dth{\leftarrowfill}}%
       }\limits
\ifx\N@one\empty\else ^{\box\z@}\fi
\ifx\N@two\empty\else _{\box\@ne}\fi}}
\def\T@fr@^#1{\@ifnextchar_ {\Tofr@^{#1}}{\Tofr@^{#1}_{}}}

\def\Two@^#1_#2{\def\N@one{#1}\def\N@two{#2}\mathrel{\setW@dth{#1}{#2}
\mathop{\vcenter{\hbox to \W@dth{\rightarrowfill}\kern-1.7ex
                 \hbox to \W@dth{\rightarrowfill}}%
       }\limits
\ifx\N@one\empty\else ^{\box\z@}\fi
\ifx\N@two\empty\else _{\box\@ne}\fi}}
\def\Tw@@^#1{\@ifnextchar_ {\Two@^{#1}}{\Two@^{#1}_{}}}

\def\to{\@ifnextchar^ {\t@@}{\t@@^{}}}
\def\from{\@ifnextchar^ {\t@@left}{\t@@left^{}}}
\def\tofro{\@ifnextchar^ {\t@fr@}{\t@fr@^{}}}
\def\To{\@ifnextchar^ {\T@@}{\T@@^{}}}
\def\From{\@ifnextchar^ {\T@@left}{\T@@left^{}}}
\def\Two{\@ifnextchar^ {\Tw@@}{\Tw@@^{}}}
\def\Tofro{\@ifnextchar^ {\T@fr@}{\T@fr@^{}}}

\makeatother

\input{diagxy}

\begin{document}
\maketitle

\begin{abstract}
We revisit the definition of Cartesian differential categories, showing that a slightly more general version is useful for a number of reasons.  As one application, we show that these general differential categories are comonadic over Cartesian categories, so that every Cartesian category has an associated cofree differential category.  We also work out the corresponding results when the categories involved have restriction structure, and show that these categories are closed under splitting restriction idempotents.  
\end{abstract}

\tableofcontents


\section{Introduction}

Cartesian differential categories \cite{cartDiff} were developed as an axiomatization of the essential properties of the derivative.  The standard example is differentiation of smooth functions between Cartesian spaces, but there are many other examples, such as differentiation of polynomials, differentiation of smooth functions between convenient vector spaces \cite{convenient-diff}, and differentiation of data types \cite{cockettFMCS2012}.  With an additional axiom, the definition gives the categorical semantics for the differential lambda calculus of \cite{diffLambda}, as described in \cite{manzonetto}.  In addition, every category with an abstract ``tangent functor'' \cite{rosicky} has an associated Cartesian differential category \cite{tangentStructure}.  For example, any model of synthetic differential geometry \cite{sdg99} has an associated Cartesian differential category.  Finally, \cite{faa} demonstrated the surprising result that there are (co)free instances of Cartesian differential categories.  

However, examined more closely, there are a number of problems with the definition of Cartesian differential category that all point to a similar root defect.  The first objection is philosophical.  In a Cartesian differential category, every map $f: X \to Y$ has an associated differential map $D[f]: X \times X \to Y$.  However, the two $X$'s in the domain of $D[f]$ play different roles.  In the canonical example of the category of smooth maps between finite-dimensional vector spaces, $D[f]$ is the Jacobian, evaluated at the second $X$, then applied in the direction of the first $X$.  In other words, we think of the first $X$ as consisting of vectors, and the second $X$ consisting of points.  The dual nature of $X$ is not reflected in the definition, and leads one to consider whether it may be possible for the two $X$'s to in fact be different objects.

A second objection occurs when one inspects the comonadicity of Cartesian differential categories.  In \cite{faa}, the authors showed that there was a comonad on left additive Cartesian categories for which certain coalgebras were Cartesian differential categories.  This leads one to wonder what the general coalgebras may be\footnote{The authors themselves note this, saying ``The more general construction...seems actually to be more natural, and this is an indication that the construction has more general forms which
we shall not explore here'' (pg. 397--398).}.  Again, the more general coalgebras point to a definition in which the domain of the derivative should be two different objects.  

The final objection occurs when one combines differential categories with restriction categories, as is done in \cite{diffRes}.  One of the most basic operations on any restriction category is to split the restriction idempotents.  Unfortunately, differential restriction categories are not closed under this operation.  Again, the problem is resolved by allowing the two elements of the domain of $D[f]$ to be seperate, so that, for example, while the derivative may only be evaluated in some open set $U \subseteq \mathbb{R}^n$, the vector along which it is taken is any vector in $\mathbb{R}^n$. 

With these considerations in mind, we reformulate the definition of Cartesian differential categories (and later, differential restriction categories).  In the new definition, not every object need have the structure of a commutative monoid.  Instead, to each object $X$ there is an assigned commutative monoid $L(X) = (L_0(X),+_X,0_X)$ which we think of as the ``object of vectors'' associated to the ``object of points'' $X$ (naturally, one of the axioms for this operation is $L(L_0(X)) = L(X)$).  The derivative of a map $f: X \to Y$ is then a map $L(X) \times X \to L(Y)$ satisfying almost identical axioms to those for Cartesian differential categories.  

Not only does this more general version solve all the problems mentioned above, it also reveals a striking new property.  In the original version, Cartesian differential categories were comonadic over Cartesian left additive categories.  In the new version, Cartesian differential categories are comonadic over Cartesian categories (that is, categories with finite products).  As nearly every naturally-occuring category has finite products, this shows that there are a vast number of Cartesian differential categories.  This is a remarkable result, considering the intricacy of the axioms, and underlines the importance of the constructions in \cite{faa}.  

The paper is laid out as follows.  We begin by giving our generalized definition of Cartesian differential categories, then show how they are the coalgebras for a slightly modified version of the \Faa \ di bruno comonad of \cite{faa}.  Fortunately, most of the work has been done in \cite{faa}; only one small modification to how one of the differential axioms is arrived at is required.  We also make a small note about the linear maps in the cofree examples.

Following this, we work out the corresponding restriction versions.  We first give the generalized version of the differential restriction categories of \cite{diffRes}, and show that unlike their ordinary counterparts, they are closed when we split the restriction idempotents.  Following this, we give a restriction version of the the \Faa \   di bruno comonad (note that this has not been done even in the non-generalized version of differential restriction categories).  There are some small subtleties that require some checking, but for the most part, the proofs are again as in \cite{faa}.  Again, however, the end result is striking: every Cartesian restriction category has an associated (generalized) differential restriction category.

All of this leads to an obvious next step: determine the nature of these cofree Cartesian differential categories, and understand how they may be used.  In particular, it may be worth understanding the associated tangent structure \cite{tangentStructure} of these examples.

\section{Cartesian differential categories revisited}

We begin by generalizing the central definition of \cite{cartDiff}.    As noted in the introduction, the main point of generalization is to allow examples where not all objects need have the structure of a commutative monoid.  Instead, each object merely has an associated commutative monoid satisfying two axioms; one thinks of this object as the ``vectors'' associated to the object.  The derivative of a map has domain taking values in the product of the object with its object of vectors.

If we have a monoid $(A,+_A,0_A)$ and maps $f,g: X \to A$, we will use $f+g$ to denote $\<f,g\>+_A$ and $0: X \to A$ to denote the map $!0_A$.  A \textbf{Cartesian category} will mean a category with chosen finite products.  

\begin{definition}
A \textbf{generalized Cartesian differential category} consists of a Cartesian category $\X$ with:
\begin{itemize}
	\item for each object $X$, a commutative monoid $L(X) = (L_0(X),+_X,0_X)$, satisfying 
		\[ L(L_0(X)) = L(X) \mbox{ and } L(X \times Y) = L(X) \times L(Y), \]
	\item for each map $f: X \to Y$, a map $D[f]: L_0(X) \times X \to L_0(Y)$ such that:

\begin{enumerate}[{\bf [CD.1]}]
\item $D(+_X) = \pi_0 +_X, D(0_X) = \pi_0 0_X$,
\item $\< a + b,c\>D[f] = \< a,b\>D[f] + \< b,c\>D[f]$ and $\< 0,a\>D[f] = 0$;
\item $D[\pi_0] = \pi_0\pi_0$, and $D[\pi_1] = \pi_0\pi_1$;
\item $D[\<f,g\>] = \< D[f],D[g]\>$;
\item $D[fg] = \<D[f],\pi_1f \>D[g]$;
\item $\<\<a,0\>,\<c,d\>\>D[D[f]] = \<a,d\>D[f]$;
\item $\< \< 0,b\>,\< c,d\>\> D[D[f]] = \<\< 0,c\>,\< b,d\>\> D[D[f]]$;
\end{enumerate}
\end{itemize}
\end{definition}

Note that only {\bf [CD.1]} has a slightly different form than given for Cartesian differential categories.  In fact, however, the definition given here is a more natural form of that axiom, as the following lemma demonstrates:

\begin{lemma}
In the presence of the other axioms, {\bf [CD.1]} is equivalent to asking that for maps $f,g: X \to L(Y)$, 
	\[ D[f+g] = D[f]+D[g] \mbox{ and } D[0]=0.  \]
\end{lemma}
\begin{proof}
Suppose we know $D(f+g) = D(f) + D(g)$ and $D(0) = 0$. Note that $+_X = \pi_0 + \pi_1$, so we have 
	\[ D(+_X) = D(\pi_0 + \pi_1) = D(\pi_0) + D(\pi_1) = \pi_0\pi_0 + \pi_0\pi_1 = \pi_0(+_X), \]
and similarly for $0$.  

Conversely, suppose we know that {\bf [CD.1]} is satisfied.  Consider:
	\[ D(f+g) = D(\<f,g\>+_X) = \<D(\<f,g\>),\pi_1\<f,g\>\>D(+_X) = \]
	\[ \<\<Df,Dg\>,\pi_1\<f,g\>\>\pi_0+_X = \<Df,Dg\>+_X = Df + Dg, \]
and similarly for the preservation of $0$.
\end{proof}

This then gives us the following:
\begin{example}
Any Cartesian differential category is a generalized Cartesian differential category, with $L(X) := (X,\pi_0 + \pi_1, 0)$.
\end{example}

For example, the category of finite dimensional vector spaces and smooth maps between them, or the category of convenient vector spaces and smooth maps between them \cite{convenient-diff} are examples. 

It is important to note that generalizing the definition in this way also allows for trivial examples:

\begin{example}
If $\X$ is a Cartesian category, then defining 
	\[ L(X) := 1 \mbox{ and } D[f] := \ !, \]
gives $\X$ the structure of a generalized Cartesian differential category.  
\end{example}
 
Applying Proposition \ref{propSplitting} gives us the following non-trivial generalized example:

\begin{example}
The category with objects $U \subseteq \mathbb{R}^n$ and smooth maps 
	\[ f: (U \subseteq \mathbb{R}^n) \to (V \subseteq \mathbb{R}^m) \]
forms a generalized Cartesian differential category, with the Jacobian as the derivative.
\end{example}
In the next section, we shall see that every Cartesian category generates a cofree generalized Cartesian differential category which is only trivial if the only monoid in $\X$ is the terminal object.

\subsection{The \Faa \ di Bruno comonad and its coalgebras}

In this section, we generalize the \Faa \ di Bruno comonad of \cite{faa}.

\begin{definition}
Let \textbf{cartCat} denote the category whose objects are Cartesian categories, and whose arrows are functors which preserve the specified products exactly.
\end{definition}

\begin{proposition}
There is an endofunctor on \textbf{cartCat}, $\mbox{\faa}$, with $\mbox{\faa}(\X)$ having:
\begin{itemize}
	\item objects pairs $((A,+,0),X)$ with $(A,+,0)$ a commutative monoid in $\X$, and $X$ an object of $\X$;
	\item a morphism from $((A,+_A,0_A),X)$ to $((B,+_B,0_B),Y)$ consists of an infinite sequence of maps $(f_*,f_1,f_2, \ldots)$ with $f_*: X \to Y$ simply a map in $X$, and for each $n$, $f_n: A \times A \times \ldots \times A \times X \to B$ is a map in $\X$ that is additive and symmetric in its first $n$ variables,
	\item composition and identities as in \cite{faa};
	\item with the product 
		\[ ((A,+_A,0_A),X) \times ((B,+_B,0_B),Y)) := ((A \times B, \mbox{ex}(+_A \times +_B), 0_A \times 0_B),X \times Y) \]
		(where ex is the map that interchanges the interior two terms) and projections
		\[ \pi_{(A,X)} := (\pi_X, \pi_0 \pi_A, 0, 0, \ldots), \pi_{(B,Y)} := (\pi_Y, \pi_0 \pi_B, 0, 0, \ldots), \]
\end{itemize}
and, given a Cartesian functor $F: \X \to \Y$, $\mbox{\faa}(F)$ has the obvious action:
\begin{itemize}
	\item $\mbox{\faa}(F)((A,+,0),X) = ((FA,F+,F0),FX)$;
	\item $\mbox{\faa}(F)(f_*,f_1,f_2, \ldots) = (F(f_*),F(f_1),F(f_2), \ldots)$,
\end{itemize}
which is well-defined since $F$ preserves the specified products.  
\end{proposition}

\begin{proof}
The proof is identitical to that in \cite{faa}.
\end{proof}

There is a natural comparison between this endofunctor and the commutative monoid endofunctor, which we now describe.  

\begin{definition}
let $\textbf{cMon}$ denote the endofunctor on $\textbf{cartCat}$ which sends a category $\X$ to its category of commutative monoids and additive maps between them (with its obvious product struture).
\end{definition}

\begin{proposition}
There is a natural transformation $\lambda: \textbf{cMon}(\X) \to \faa(\X)$, which maps
	\[ (A,+,0) \mapsto ((A,+,0),A) \]
and
	\[ A \to^f B \mapsto (f,\pi_0f, 0, 0, \ldots). \]
\end{proposition}
\begin{proof}
$\lambda_{\X}(f)$ is a valid map in $\faa(\X)$ since $f$ is additive.  For each $\X$, $\lambda_{\X}$ is a functor since $1_{((A,+,0),A)} = (1,\pi_0, 0, \ldots)$ and
\begin{eqnarray*}
\lambda_{\X}(f)\lambda_{\X}(g) & = & (f,\pi_0f, 0, \ldots)(g, \pi_0g, 0, \ldots) \\
& = & (fg, \<\pi_0f, \pi_1f\>\pi_0g, 0, \ldots) \\
& = & (fg, \pi_0fg, 0, \ldots) \\
& = & \lambda_{\X}(fg).
\end{eqnarray*}
It preserves products since the projections in  $\faa(\X)$ are $\lambda_{\X}(\pi_0)$, $\lambda_{\X}(\pi_1)$.  For naturality, for a product-preserving functor $F: \X \to \Y$ we need the diagram
\[
\bfig
	\square<1000,500>[\textbf{cMon}(\X)`\faa(\X)`\textbf{cMon}(\Y)`\faa(\Y);\lambda_{\X}`\textbf{cMon}(F)`\faa(F)`\lambda_{\Y}]
\efig
\]
to commute.  It is easy to see that on objects these two composite functors are equal, while for an addition-and-0-preserving map $f: (A,+_A,0_A) \to (B,+_B, 0_B) \in \X$, 
	\[ \lambda_{\Y}(\textbf{cMon}(F)(f)) = \lambda_{\Y}(F(f)) = (F(f),\pi_0 F(f),F(0_B),\ldots) \]
(since the $0$ in the codomain is $F(0_B)$), while
	\[ \faa(F)(\lambda_{\X}(f)) = \faa(F)(f,\pi_0f,0_B,\ldots) = (F(f),\pi_0 F(f),F(0_B), \ldots) \]
since $F$ preserves the specified products.  So $\lambda$ is natural, as required.  
\end{proof}

\begin{corollary}
If $(A,+,e)$ is a commutative monoid in $\X$, then 
	\[ ((A,+,0),A),(+,\pi_0+,0,\ldots),(0,\pi_0 0, 0,\ldots)) \]
is a commutative monoid in $\faa(\X)$.
\end{corollary}
\begin{proof}
Since $(A,+,0)$ is commutative, $(A,+,0)$ is a commutative monoid in $\textbf{cMon}(\X)$, and so gets sent by $\lambda_{\X}$ to a commutative monoid in $\faa(\X)$.
\end{proof}

\begin{theorem}
$\mbox{\faa}$ has the structure of a comonad, with counit $\epsilon: \mbox{\faa}(\X) \to \X$ given by:
\begin{itemize}
	\item $\epsilon((A,+,0),X) = X$,
	\item $\epsilon(f_*,f_1,f_2,\ldots) = f_*$,
\end{itemize}
and comultiplication $\delta: \faa(\X) \to \faa^2(\X)$ given by:
\begin{itemize}
	\item $\delta((A,+,0),X) = (((A,+,0),A),(+,\pi_0+,0,\ldots),(0,\pi_0 0,0,\ldots)),((A,+,0),X)))$,
	\item action on arrows as in \cite{faa}.
\end{itemize}
\end{theorem}
\begin{proof}
Again, the hard work has been done in \cite{faa}.  The only thing extra needed to check here is that $(((A,+,0),A),(+,\pi_0+,0,0,\ldots),(0,\pi_0 0,),0,0,\ldots),(A,+,0),X)$ is an object of $\faa^2(\X)$, and this was done in the previous corollary. 
\end{proof}

\begin{theorem}
The coalgebras for the comonad $\faa$ are exactly the generalized Cartesian differential categories.
\end{theorem}
\begin{proof}
Again, most of this is as in \cite{faa}.  If we have a coalgebra ${\cal D}: \X \to \faa(\X)$, we let ${\cal D}(X) = ({\cal D}_0(X),{\cal D}_1(X))$.  Since ${\cal D}$ satisfies the counit equations, we must have ${\cal D}_1(X) = X$.  We define $L(X) := {\cal D}_0(X)$, and $D[f] := [{\cal D}(f)]_1$.  Since ${\cal D}$ preserves products, we have $L(X \times Y) = L(X) \times L(Y)$.  

Writing $L(X)$ as $(L_0(X), +_X, 0_X)$, the coassociativity equation
\[
	\bfig
	\square<750,375>[\X`\faa(\X)`\faa(\X)`\faa^2(\X);{\cal D}`{\cal D}`\delta`\faa({\cal D})]
	\efig
\]
on objects tells us that 
	\[ (((L(L_0(X), L_0(X)), {\cal D}(+_X), {\cal D}(e_X), (L(X), X)) \]
	\[ = (((L(X),L_0(X)), (+_X, \pi_0+_X, 0, \ldots), (0_X, \pi_0 0_X, 0, \ldots)), (L(X), X)), \]
so that we get
	\[ L(L_0(X)) = L(X), D(+_X) = \pi_0+_x, \mbox{ and } D(0_X) = \pi_0 0_x.  \] 
The equations \cd{2}--\cd{7} follow exactly as in \cite{faa}.  

Conversely, if we have a generalized Cartesian differential category, we define 
	\[ {\cal D}(X) := (L(X), X)) \]
and
	\[ {\cal D}(f) := (f,D(f), D_2(f), D_3(f), \ldots) \]
where
	\[ D_n(f) := \<0, 0, \ldots 0, \pi_0, \pi_1, \ldots \pi_n\>D^n(f). \]
Almost all of the work in showing that this is a coalgebra is done in $\cite{faa}$; the only thing left to check is that ${\cal D}(+) = (+, \pi_0+, 0, \ldots)$ and ${\cal D}(0) = (0, \pi_0 0, 0, \ldots)$.  But \cd{1} gives ${\cal D}(+)_1 = \pi_0+$, and the higher terms are then 0, as 
	\[ \<0,\pi_0, \pi_1,\pi_2\>D^2(+) = \<0,\pi_0, \pi_1,\pi_2\>D(\pi_0+) = \<0,\pi_0, \pi_1,\pi_2\>\pi_0\pi_0D(+) = 0 \] 
and similarly for ${\cal D}(0)$.
\end{proof}

\begin{corollary}
If $\X$ is a Cartesian category, then $\faa(\X)$ is a generalized Cartesian differential category, with
	\[ D(f) = [\delta(f)]_1. \]
\end{corollary}

Of course, this is nothing more than stating that cofree coalgebras exist, but it is worth highlighting this particular result, as it shows that there are innumerable examples of generalized Cartesian differential categories.  Note that $\faa(\X)$ has trivial differential structure if and only if $1$ is the only commutative monoid in $\X$, as for example happens if $\X$ is a poset with finite meets.  But in most cases of interest (say, $\X = \textbf{sets}$), $\faa(\X)$ is highly non-trivial.

A more in-depth investigation of such cofree generalized Cartesian differential categories is clearly required; for now, we content ourselves with determining their linear maps.  

In a Cartesian differential category, a map $f: X \to Y$ is called linear if $D(f) = \pi_0 f$.  For a general map in a generalized Cartesian differential category, this is not possible, as $\pi_0 f$ is not even well-defined.  But if $L_0(X) = X$ and $L_0(Y) = Y$, then the types do match, and we can define what it maps for such maps to be linear.  

\begin{definition}
Say that an object $X$ in a generalized Cartesian differential category is a \textbf{linear object} if $L_0(X) = X$.  Say that map $f: X \to Y$ between linear objects is a \textbf{linear map} if $D(f) = \pi_0 f$. 
\end{definition}

We can now determine the linear maps in cofree generalized Cartesian differential categories.

\begin{proposition}
The linear maps in $\faa(\X)$ are exactly those maps of the form $\lambda(f)$ for $f$ an additive map from $(A,+_A,0_A)$ to $(B,+_B,0_B)$.
\end{proposition}
\begin{proof}
Note that the linear objects in $\faa(\X)$ are those objects of the form 
	\[ \lambda(A,+_A,0_A) = ((A,+_A,0_A),A). \]
Now suppose
	\[ ((A,+_A,0_A,A) \to^{(f_*,f_1, f_2, \ldots)} ((B,+_B,0_B),B) \]
is a linear map.  We want to show that $f = \lambda(f_*)$, ie., that 
	\[ f_1 = \pi_0 f_* \mbox{ and } f_n = 0 \ \forall n \geq 2. \]
Since $f$ is linear, we have 
	\[ (\pi_0 f)_* = (D(f))_* = (\delta(f)_1)_*. \]
But $(\pi_0 f)_* = \pi_0f_*$ and from the definition of $\delta$, $(\delta(f)_1)_* = f_1$, so we have $f_1 = \pi_0 f_*$.  

We shall now prove that for all $n \geq 2$, $f_n = 0$ by induction on $n$.  For the case $n=2$, we know that $(\pi_0 f)_1 = (\delta(f)_1)_1$, so by the definition of $\delta$ and composition in $\faa(X)$, we have
	\[ \<\pi_0 \pi_0, \pi_1\pi_0\>f_1 = \<\pi_0\pi_1, \pi_1\pi_0,\pi_1\pi_1\>f_2 + \<\pi_0\pi_0, \pi_1\pi_1\>f_1. \]
In particular, in a context $\<a,b,c,x\>$, we have
	\[ \<a,c\>f_1 = \<b,c,x\>f_2 + \<a,x\>f_1. \]
Then setting $a = 0$ and recalling that $f_1$ is linear in its first variable, we have
	\[ 0 = \<b,c,x\>f_2 \]
For any $b,c,x$.  Thus $f_2 = 0$.  

For $n > 2$, we also have $(\pi_0 f)_{n-1} = (\delta(f)_1)_{n-1}$.  For the left side, by the definition of composition in $\faa(\X)$, $(\pi_0 f)_{n-1}$ is a sum over certain binary trees $\tau$. But with the exception of the binary tree with a single node out of the root, the expression $(\pi_0 \star f)_{n-1}$ will $f_i$ applied to at least one term with $0$; since each $f_i$ is additive, these expression are then $0$.  For the tree with a single node out of the root, we have $f_{n-1}$ applied to some terms.  But by the induction assumption, $f_{n-1}$ is 0, so we have $(\pi_0 f)_{n-1} = 0$.  

For the right side, recalling the definition of $(\delta(f)_1)_{n-1}$ from \cite{faa}, the only possible choices for the index $s$ are $0$ or $1$ (as in this case $r=1$).  $(\delta(f)_1)_{n-1}$ is then a sum of terms, one of which is $f_n$, the other terms $f_{n-1}$ applied at some value.  However, by the induction assumption, $f_{n-1} = 0$, so $(\delta(f)_1)_{n-1} = f_n$.  Putting this together with the above gives $f_n = 0$, as required.

We have thus shown that if $f$ is linear, then $f$ must be of the form $\lambda(f_*)$; conversely, the above calculations also show that maps of such form are linear.
\end{proof}


\section{Differential restriction categories revisited}\label{sectionDiff}

As the first part of this paper generalized the Cartesian differential categories of \cite{cartDiff}, so this paper generalizes the differential restriction categories of \cite{diffRes}.  The immediate benefit of the generalized version is that, unlike with ordinary differential restriction categories, splitting the restriction idempotents of a differential restriction category retains differential structure.  We shall also describe the restriction version of the \Faa di bruno comonad, an aspect that has not been explored in even the non-generalized version.  

\begin{definition}
A \textbf{generalized differential restriction category} is a Cartesian restriction category with: \begin{itemize}
	\item for each object $X$, a total commutative monoid $L(X) = (L_0(X),+_X,0_X)$, satisfying 
		\[ L(L_0(X)) = L(X) \mbox{ and } L(X \times Y) = L(X) \times L(Y), \]
	\item for each map $f: X \to Y$, a map $D[f]: L_0(X) \times X \to L_0(Y)$ such that:
		\begin{enumerate}[{\bf [DR.1]}]
		\item $D[+_X] = \pi_0 +_X$ and $D[0_X]=\pi_0 0_X$;
		\item $\< a+b,c\>D[f] = \< a,b\>D[f] + \< b,c\>D[f]$ and $\< 0,a\>D[f] = \rst{af}0$;
		\item $D[\pi_0] = \pi_0\pi_0$, and $D[\pi_1] = \pi_0\pi_1$;
		\item $D[\<f,g\>] = \< D[f],D[g]\>$;
		\item $D[fg] = \<D[f],\pi_1f \>D[g]$;
		\item $\<\<a,0\>,\<c,d\>\>D[D[f]] = \rst{c} \<a,d\>D[f]$;
		\item $\< \< 0,b\>,\< c,d\>\> D[D[f]] = \<\< 0,c\>,\< b,d\>\> D[D[f]]$;
		\item $D[\rst{f}] = (1 \times \rst{f})\pi_0 = \rs{\pi_1 f}\pi_0$;
		\item $\rst{D[f]} = 1 \x \rst{f} = \rs{\pi_1 f}$.
		\end{enumerate}
\end{itemize}
\end{definition}

We recall a number of examples.  

\begin{example}
Any generalized Cartesian differential category is a generalized differential restriction category, when equipped with the trivial restriction structure ($\rs{f} = 1$ for all $f$).
\end{example}

\begin{example}
Any differential restriction category is a generalized differential restriction category, with $L(X) = X$ for each $X$.
\end{example}

The standard example of a differential restriction category is:

\begin{example}
Smooth functions defined on open subsets of ${\cal R}^n$.
\end{example}

More examples, such as differential restriction categories of rational functions, can be found in \cite{diffRes}.

Recall that if $\X$ is a restriction category, then the restriction idempotent splitting of $\X$, $K_r(\X)$, is a restriction category with:
\begin{itemize}
	\item objects restriction idempotents $(X, e = \rs{e})$;
	\item a map $f: (X,e_1) \to (Y,e_2)$ is a map $f: X \to Y$ such that $e_1fe_2 = f$;
	\item restriction and composition as in $\X$, and $1_{(X,e)} := e$.
\end{itemize}

One problem with differential restriction categories is that even if $\X$ is a differential restriction category, $K_R(\X)$ need not be: because of $\bf [DR.9]$, the derivative must be total in the first variable, and so the derivative of a map $f: (X,e_1) \to (Y,e_2)$ cannot have domain $(X,e_1) \times (X,e_1)$.  With generalized differential restriction categories, this is no longer a problem, as we can set $L(X,e_1) = (L(X),1)$.  

\begin{proposition}\label{propSplitting}
If $\X$ is a generalized differential restriction category then $K_r(\X)$ is also, with 
	\[ L(X,e) := (L(X),1) \mbox{ and } D(f) := D(f).  \]
\end{proposition}
\begin{proof}
The only thing to check is that $D(f)$ is a valid map in the restriction idempotent splitting category.  Suppose $f: (X,e_1) \to (Y,e_2)$.  We are claiming that $D(f)$ is a valid map from $(L_0(X) \times X, 1 \times e_1)$ to $(L_0(Y), 1)$.  So consider
	\[ (1 \times e_1)D(f) = (1 \times e_1)\rs{D(f)}D(f) = (1 \times e_1)(1 \times \rs{f})D(f) = (1 \times \rs{f})D(f) = D(f) \]
by $\bf [DR.9]$ and the fact that $e_1fe_2 = f$.  
\end{proof}

\begin{corollary}
If $\X$ is a differential restriction category, then the total maps of $K_r(\X)$ form a generalized Cartesian differential category.
\end{corollary}

As noted in the introduction, this shows that the categories whose objects are open subsets of $\mathbb{R}^n$, and whose maps are smooth maps between them, forms a generalized Cartesian differential category.



\subsection{\Faa \ di bruno - restriction version}

In this section, we expand the construction of the previous section to work with restriction categories.  To show that there is a version of $\faa$ suitable for restriction categories, we will need to begin by recalling the definition of composition in $\faa(\X)$ from \cite{faa} (section 2.1).  Using the notation of that section, recall that 
	\[ (fg)_n = \sum (f \star g)_{\tau}, \] 
where the sum is over each tree $\tau$ of length 2 and width $n$.  For such a tree $\tau$, the term $(f \star g)_{\tau}$ term is of the form
	\[ \<p_1f_{i_1}, p_2f_{i_2}, \ldots p_k f_{i_k}, \pi_n f_*\>g_k \]
where $k$, each $i_j$, and $p_i$ all depend on the the tree $\tau$, and each $p_i$ is of the form $\<\pi_{j_1}, \pi_{j_2}, \ldots \pi_{j_{i_1}}, \pi_n\>$ (see \cite{faa} for the exact details). 

\begin{theorem}\label{faaResVersion}
Given a Cartesian restriction category $\X$, there is a Cartesian restriction category $\faa(\X)$ which has:
\begin{itemize}
	\item objects pairs $((A,+,0),X)$, where $X$ is an object of $\X$ and $(A,+,0)$ is a (total) commutative monoid in $\X$;
	\item maps sequences
		\[ (f_*, f_1, f_2, \ldots): (A,X) \to (B,Y) \]
	where $f_*: X \to Y$, $f_n: (A)^n \times X \to B$, such that for each $i$, 
		\[ \rs{f_n} = 1 \times \rs{f_*} = \rs{\pi_nf_*}, \] 
	and each $f_n$ is additive and symmetric in the first $n$ variables
	\item composition and identities are defined as in the total case;
	\item restriction given by
		\[ \rs{(f_*, f_1, f_2, \ldots)} := (\rs{f_*}, \rs{\pi_1 f_*}\pi_0, \rs{\pi_2 f_*}0, \rs{\pi_3 f_*}0\ldots). \]
\end{itemize}
\end{theorem}

As is usual when determining if a category with an operation $\rs{f}$ is a restriction category, it is helpful to first determine what a map of the form $\rs{f}g$ looks like.

\begin{lemma}\label{lemmaResComposite}
With the above definition of $\rs{f}$, for $n \geq 1$ and maps $f: X \to A$, $g: X \to B$, we have
	\[ (\rs{f}g)_n = \rs{\pi_n f_*}g_n. \]
\end{lemma}
\begin{proof}
Recalling the definition of composition as above, we have $(\rs{f} \star g)_{\tau}$ is of the form 
	\[ \<p_1\rs{f}_{i_1}, p_2\rs{f}_{i_2}, \ldots p_k \rs{f}_{i_k}, \pi_n \rs{f}_*\>g_k \]
But for any $i \geq 1$, the expression $p_j\rs{f}_{i_j}$ is a restriction of $0$, and with the exception of the tree with $n$ branches out of the root, that expression occurs at least once.  Then since each $g_k$ is additive in each of the first $n$ variables, we have 
	\[ (\rs{f} \star g)_{\tau} = \rs{\pi_n \rs{f_*} g_*} 0 = \rs{\pi_n f_*} \rs{\pi_n g_*} 0. \]
	
For that one tree $\tau_0$ with $n$ branches out of the root, 
	\[ (\rs{f} \star g)_{\tau_0} = \<\<\pi_0, \pi_n\>(\rs{f})_1, \<\pi_1,\pi_n\>(\rs{f})_1, \ldots \<\pi_{n-1}, \pi_n\>(\rs{f})_1, \pi_n f_*\>g_n. \]
But for each $i \in \{1 \ldots n-1\}$,
	\[ \<\pi_i, \pi_n\>(\rs{f})_1 = \<\pi_i, \pi_n\>\rs{\pi_1 f_*}\pi_0 = \rs{\pi_n f_*} \<\pi_i, \pi_n\>\pi_0 = \rs{\pi_n f_*}\pi_i, \]
so that
	\[ (f \star g)_{\tau_0} = \rs{\pi_n f_*} \<\pi_0, \pi_1, \ldots \pi_n\>g_n = \rs{\pi_nf_*}g_n \]
Thus, the sum over all trees $\tau$ equals
	\[ \rs{\pi_nf_*}g_n + \rs{\pi_nf_*} \rs{\pi_n g_*}0 = \rs{\pi_n f_*} \rs{\pi_n g_*}g_n = \rs{\pi_nf_*}g_n \]
by assumption on $g_n$.  Thus we have
	\[ (\rs{f}g)_n = \rs{\pi_nf_*}g_n, \]
as required.
\end{proof}

We are now in a position to prove Theorem \ref{faaResVersion}. \\

\begin{proof}

We first need to check that the identities, composites, and restriction satisfy the added requirement on the restriction of its components.  That identities satisfy the requirement is obvious, since $1_{(A,x)} = (1,\pi_0,0_A, \ldots)$ and both $\pi_0$ and $0_A$ are total.  

To check that the composite of two maps $f: (A,X) \to (B,Y)$, $g: (B,Y) \to (C,Z)$ satisfies the restriction requirement, as above, recall that $(fg)_n = \sum (f \star g)_{\tau}$. Since $\rs{x + y} = \rs{x}\rs{y}$, to calculate $\rs{(fg)_n}$, we need to calculate each $\rs{(f \star g)_{\tau}}$.  For such a tree $\tau$, this term is of the form
	\[ \<p_1f_{i_1}, p_2f_{i_2}, \ldots p_k f_{i_k}, \pi_n f_*\>g_k \]
where $k$, each $i_j$, and $p_i$ all depend on the the tree $\tau$.  In particular, however, each $p_m$ is of the form $\<\pi_{j_1}, \pi_{j_2}, \ldots \pi_{j_{i_m}}, \pi_n\>$ (where again each $j_l$ depends on $\tau$) so we have
\begin{eqnarray*}
&   & \rs{p_m f_{i_m}} \\
& = & \rs{\<\pi_{j_1}, \pi_{j_2}, \ldots \pi_{j_{i_m}}, \pi_n\>f_{i_m}} \\
& = & \rs{\<\pi_{j_1}, \pi_{j_2}, \ldots \pi_{j_{i_m}}, \pi_n\>\rs{f_{i_m}}} \\
& = & \rs{\<\pi_{j_1}, \pi_{j_2}, \ldots \pi_{j_{i_m}}, \pi_n\>\rs{\pi_{i_m} f_*}}  \\
& = & \rs{\<\pi_{j_1}, \pi_{j_2}, \ldots \pi_{j_{i_m}}, \pi_n\>\pi_{i_m} f_*} \\
& = & \rs{\pi_n f_*} \\
\end{eqnarray*}

Then we can calculate, for any such tree $\tau$,

\begin{eqnarray*}
\rs{(f \star g)_{\tau}} & =  & \rs{\<p_1f_{i_1}, p_2f_{i_2}, \ldots p_k f_{i_k},\pi_n f_*\>g_k} \\
& =  & \rs{\<p_1f_{i_1}, p_2f_{i_2}, \ldots p_k f_{i_k},\pi_n f_*\>\rs{g_k}} \\
& =  & \rs{\<p_1f_{i_1}, p_2f_{i_2}, \ldots p_k f_{i_k},\pi_n f_*\>\rs{\pi_n g_*}}  \\
& =  & \rs{\<p_1f_{i_1}, p_2f_{i_2}, \ldots p_k f_{i_k},\pi_n f_*\>\pi_n g_*} \\
& =  & \rs{\rs{p_1f_{i_1}} \rs{p_2f_{i_2}} \ldots \rs{p_k f_{i_k}}\pi_n f_* g_*} \\
& =  & \rs{\rs{\pi_n f_*} \rs{\pi_n f_*} \ldots \rs{\pi_n f_*}\pi_n f_* g_*}  \\
& = & \rs{\pi_n f_*g_*} \\
& = & \rs{\pi_n (fg)_*} \\
\end{eqnarray*}
so that we get $\rs{(fg)_n} = \rs{\pi_n (fg)_*}$, as required.  

Each restriction map satisfies the requirement on the restriction of its components since 
	\[ \rs{\rs{\pi_n f_*} \pi_0} = \rs{\pi_n}f_* = \rs{\rs{\pi_n f_*} 0} \]
as $0$ and $\pi_0$ are both total.  

That $\faa(\X)$ is a category is as in \cite{faa}.  We now turn to checking the restriction axioms.  As the equality of $*$ terms follows directly, we will simply check for the $n \geq 1$ terms.  For {[\bf R.1]}, by lemma \ref{lemmaResComposite},
	\[ (\rs{f}f)_n = \rs{\pi_n f_*} f_n = f_n \]
by assumption on $f_n$.  For {[\bf R.2]}, for $n \geq 2$, by lemma \ref{lemmaResComposite}, we have
	\[ (\rs{f}\rs{g})_n = \rs{\pi_n f_*} \rs{ \pi_n g_*} 0 = \rs{\pi_n g_*} \rs{ \pi_n f_*} 0 = (\rs{g}\rs{f})_n \]
and $n = 1$ is similar.  For {[\bf R.3]}, again by lemma \ref{lemmaResComposite}, we have
	\[ (\rs{\rs{f}g})_1 = \rs{\pi_1 (\rs{f}g)_*} 0 = \rs{\pi_1 \rs{f_*}g} 0 = \rs{\rs{\pi_1 f_*} \pi_1 g} 0 = \rs{\pi_1 f_*} \rs{\pi_1 g_*} 0 \]
which is the required value, by the calcuation for {[\bf R.2]}; again $n=1$ is similar. 

For {[\bf R.4]}, we need to find the $n$th term of $f\rs{g}$.  This time, for any tree $\tau$ with the exception of the tree with a \emph{single} node out of the root, we have
\begin{eqnarray*}
(f \star \rs{g})_{\tau} & = & \<p_1 f_{i_1}, p_2 f_{i_2}, \ldots \pi_n f_*\>\rs{\pi_n g_*} 0 \\
& = & \rs{\<p_1 f_{i_1}, p_2 f_{i_2}, \ldots \pi_n f_*\>\pi_n g_*} \<p_1 f_{i_1}, p_2 f_{i_2}, \ldots \pi_n f_*\>\ 0 \\
& = & \rs{\pi_n f_* g_*} \rs{\pi_n f_*} 0 \\
& = & \rs{\pi_n f_* g_*} 0
\end{eqnarray*}
For that one tree $\tau_1$ with a single node out of the root, we have
	\[ (f \star \rs{g})_{\tau_1} = \<f_n, \pi_n f_*\>\rs{\pi_1 g_*} \pi_0 = \rs{\pi_n f_* g_n} f_n.  \]
Then summing over all trees $\tau$ gives
	\[ (f \rs{g})_n = \rs{\pi_n f_* g_*} 0 + \rs{\pi_n f_* g_n} f_n = \rs{\pi_n f_* g_n} f_n.  \]
Conversely, by lemma \ref{lemmaResComposite}, we have
	\[ (\rs{fg}f)_n = \rs{\pi_n (fg)_*} f_n = \rs{\pi_n f_* g_*} f_n \]
so that {[\bf R.4]} is satisfied, as required.  

\end{proof}

\begin{proposition}\label{faaInequalities}
In the restriction category $\faa(\X)$, for maps $f,g: (X,A) \to (Y,B)$:
\begin{enumerate}[(i)]
	\item $f$ is total if and only if $f_*$ is total;
	\item $f \leq g$ if and only if $f_* \leq g_*$ and $f_n \leq g_n$ for each $n \geq 1$;
	\item $f \smile g$ if and only if $f_* \smile g_*$ and $f_n \smile g_n$ for each $n 
	\geq 1$.
\end{enumerate}
\end{proposition}

\begin{proof}
\begin{enumerate}[(i)]
	\item If $f$ is total, then in particular $(\rs{f})_* = 1$, so $f_*$ is total.  Conversely, if $f_*$ is total, then for $n \geq 2$, 
		\[ (\rs{f})_n = \rs{\pi_n f_*} 0 = \rs{\pi_n} 0 = 0 \]
	and similarly $(\rs{f})_1 = \pi_0$, so that $f$ is total.  
	\item Recall that $f \leq g$ means $\rs{f}g = f$.  So $f \leq g$ if and only if $f_* \leq g_*$ and for each $n \geq 1$, $(\rs{f}g)_n = f_n$.  But by lemma \ref{lemmaResComposite},
		\[ (\rs{f}g)_n = \rs{\pi_n f_*} g_n = \rs{f_n} g_n \]
	by assumption on $f_n$.  Thus $f \leq g$ if and only if $f_* \leq g_*$ and for each $n \geq 1$, $f_n \leq g_n$. 
	\item Recall that $f \smile g$ means $\rs{f}g = \rs{g}f$.  The result then follows as in (ii).
\end{enumerate}
\end{proof}
	
\begin{proposition}
With product structure as in the total case:
	\[ \pi_i = (\pi_i, \pi_0\pi_i, 0, 0, \ldots), \<f,g\>_n := \<f_n, g_n\>, \]
$\faa(\X)$ is a Cartesian restriction category.
\end{proposition}
\begin{proof}
We first need to show $\<f,g\>\pi_0 \leq f$, so consider the term $(\<f,g\>\pi_0)_n$.  As in the proof of {[\bf R.4]} in Theorem \ref{faaResVersion}, for any tree $\tau$ with the exception of the the tree $\tau_1$ which has a single node coming out of the root, $(\<f,g\> \star \pi_0)_{\tau}$ is of the form
\begin{eqnarray*}
&   & \<p_1 \<f,g\>_{i_1}, p_2 \<f,g\>_{i_2}, \ldots \pi_n \<f,g\>_*\>0 \\
& = & \rs{p_1 \<f,g\>_{i_1}} \rs{p_2 \<f,g\>_{i_2}} \ldots \rs{\pi_n\<f,g\>_*} 0 \\
& = & \rs{\pi_n\<f_*,g_*\>} 0 \\
& = & \rs{\pi_nf_*} \rs{\pi_n g_*} 0
\end{eqnarray*}
while for $\tau_1$,
\begin{eqnarray*}
(\<f,g\> \star \pi_0)_{\tau_1} & = & \<(\<f,g\>_n, \pi_n \<f,g\>_*\>\pi_0 \pi_0 \\
& = & \<\<f_n, g_n\>, \pi_n \<f_*, g_*\>\>\pi_0 \pi_0 \\
& = & \rs{g_n} \rs{\pi_n f_*} \rs{\pi_n g_*} f_n \\
& = & \rs{g_n}f_n \mbox{ (by assumption on $g_n$)} 
\end{eqnarray*}
Thus
	\[ (\<f,g\>\pi_0)_n = \rs{\pi_nf_*} \rs{\pi_n g_*} 0 + \rs{g_n}f_n = \rs{g_n}f_n, \]
so that by lemma \ref{faaInequalities}, $\<f,g\>\pi_0 \leq f$, as required; $\pi_1$ is similar. 

We also need to show that $\rs{\<f,g\>} = \rs{f}\rs{g}$.  For $n \geq 2$, consider
	\[ (\rs{\<f,g\>})_n = \rs{\pi_n \<f_g\>_*} 0 = \rs{\pi_n \<f_*, g_*\>} 0 = \rs{\pi_n f_*} \rs{\pi_n g_*} 0 \]
while by lemma \ref{lemmaResComposite},
	\[ (\rs{f}\rs{g})_n = \rs{\pi_n f_*} (\rs{g})_n = \rs{\pi_n f_*} \rs{\pi_n g_*} 0, \]
and $n=1$ is similar, so that $\rs{\<f,g\>} = \rs{f}\rs{g}$, as required.
\end{proof}

\begin{proposition}
$\faa$ extends to an endofunctor on the category of Cartesian restriction categories (where the maps are those functors which preserve products and restrictions on the nose), where we define
	\[ \faa(F)(f_*, f_1, f_2, \ldots) := (F(f_*), F(f_1), F(f_2), \ldots) \]
\end{proposition}
\begin{proof}
The only thing to check is that $\faa{F}(f)$ satisfies the restriction requirement on its components:
	\[ \rs{F(f_n)} = F(\rs{f_n}) = F(\rs{\pi_n f_*}) = \rs{\pi_n F(f_*)} \]
since $F$ preserves restrictions and products exactly.
\end{proof}

\begin{proposition}
The endofunctor $\faa$ has the structure of a comonad, with counit $\epsilon: \mbox{\faa}(\X) \to \X$ given by:
\begin{itemize}
	\item $\epsilon((A,+,0),X) = X$,
	\item $\epsilon(f_*,f_1,f_2,\ldots) = f_*$,
\end{itemize}
and comultiplication $\delta: \faa(\X) \to \faa^2(\X)$ given by:
\begin{itemize}
	\item $\delta((A,+,0),X) = (((A,+,0),A),(+,\pi_0+,0,\ldots),(0,\pi_0 0,0,\ldots)),((A,+,0),X)))$,
	\item action on arrows as in \cite{faa}.
\end{itemize}
\end{proposition}
\begin{proof}
There are only a few additional things to check here:
\begin{enumerate}[(i)]
	\item that $\epsilon$ preserves restrictions;
	\item that $\delta(f)$ is a valid arrow in $\faa^2(\X)$;
	\item that $\delta$ preserves restrictions.
\end{enumerate}
The first part is obvious, as by definition $(\rs{f})_* = \rs{f_*}$.  

For (ii), we need to show that 
	\[ \rs{\delta(f)_n} = \rs{\pi_n \delta(f)_*}, \]
so we need to show that they are equal in each component.  For $m \geq 2$, we have
\begin{eqnarray*}
&   & (\rs{\delta(f)_n)}_m \\
& = & \rs{\pi_m (\delta(f)_n)_*} 0 \mbox{ (by definition of restriction)} \\
& = & \rs{\pi_m f_n} 0 \mbox{ (by definition of $\delta(f)$)} \\
& = & \rs{\pi_m \rs{f_n}} 0 \\
& = & \rs{\pi_m \rs{\pi_n f_*}} 0 \mbox{ (by assumption on $f_n$)} \\
& = & \rs{\pi_m \pi_n f_*} 0
\end{eqnarray*}
while
	\[ (\rs{\pi_n \delta(f)_*})_m = \rs{\pi_m (\pi_n \delta(f)_*)_*} 0 = \rs{\pi_m \pi_n f_*} 0 \]
by definition of $\delta(f)$, so that they are equal, as required.  The case $m=1$ and the $*$ component are similar.  

For (iii), we need $\delta(\rs{f}) = \rs{\delta(f)}$, so in particular we need for each $n, m \geq 1$,
	\[ ((\delta(\rs{f}))_n)_m = ((\rs{\delta(f)})_n)_m. \]
For $n, m \geq 2$, starting with the right side, we have
\begin{eqnarray*}
&   & ((\rs{\delta(f)})_n)_m \\
& = & (\rs{\pi_n \delta(f)_*}0)_m \mbox{ (by definition of restriction)} \\
& = & \rs{\pi_m(\pi_n\delta(f)_*)_*}0 \mbox{ (by lemma \ref{lemmaResComposite})} \\
& = & \rs{\pi_m \pi_n f_*} 0 \mbox{ (by definition of $\delta$).}
\end{eqnarray*}
For $((\delta(\rs{f}))_n)_m$, recalling the definition of $\delta$ from Theorem 2.2.2 in \cite{faa}, we see that $((\delta(\rs{f}))_n)_m$ is a sum of terms of the form
	\[ \<\pi_{\alpha_1, \beta_1}, \pi_{\alpha_2, \beta_2}, \ldots, \pi_m\>\rs{\pi_n f_*} 0, \]
where the indices $\alpha_i$, $\beta_j$ are given by a formula in \cite{faa}.  The particular form of these indices is not important however, as in each case we have
\begin{eqnarray*}
&   & \<\pi_{\alpha_1, \beta_1}, \pi_{\alpha_2, \beta_2}, \ldots, \pi_m\>\rs{\pi_n f_*} 0 \\
& = & \rs{\<\pi_{\alpha_1, \beta_1}, \pi_{\alpha_2, \beta_2}, \ldots, \pi_m\>\pi_n f_*} \<\pi_{\alpha_1, \beta_1}, \pi_{\alpha_2, \beta_2}, \ldots, \pi_m\> 0 \\
& = & \rs{\pi_m \pi_n f_*} 0 \mbox{ (since projections are total)}
\end{eqnarray*}
so that the sum is also $\rs{\pi_m \pi_n f_*} 0$, and we have $((\delta(\rs{f}))_n)_m = ((\rs{\delta(f)})_n)_m$.  The cases for $m,n$ equalling $1$ or $*$ or similar, so $\delta(\rs{f}) = \rs{\delta(f)}$, as required.    
\end{proof}

\begin{theorem}
The coalgebras for the comonad $(\faa,\epsilon,\delta)$ are exactly the generalized differential restriction categories.
\end{theorem}
\begin{proof}
As before, if we have a coalgebra ${\cal D}: \X \to \faa(\X)$, we let ${\cal D}(X) = ({\cal D}_0(X),{\cal D}_1(X))$.  Since ${\cal D}$ satisfies the counit equations, we must have ${\cal D}_1(X) = X$.  We define $L(X) := {\cal D}_0(X)$, and $D[f] := [{\cal D}(f)]_1$.  

For the most part, the fact that this operation satisfies the differential restriction axioms is exactly as before, with a few minor modifications.  Since $D[f] := [{\cal D}(f)]_1$ is a map in $\faa(\X)$, we must have $\rs{D(f)_1} = \rs{\pi_1 {\cal D}(f)_*} = \rs{\pi_1 f}$.  Since $\cal D$ preserves restrictions, we have 
	\[ D(\rs{f}) = {\cal D}(\rs{f})_1 = (\rs{{\cal D}(f)})_1 = \rs{\pi_1 f}\pi_0.  \]
Thus, we have both of the added differential restriction axioms.

Note that $D(f)$ being additive in its first variable means that 
	\[ \<0,x\>D(f) = \rs{\<0,x\>[{\cal D}(f)]_1} 0 = \rs{xf_*} 0 = \rs{xf}0, \] 
giving $\dr{2}$.  Similarly, as on pg. 414 of \cite{faa}, we get $\dr{6}$ by setting a certain term equal to $0$; the extra restriction term then comes out when we project.

Conversely, if we have a generalized Cartesian differential category, and define 
	\[ {\cal D}(f) := (f,D(f), D_2(f), D_3(f), \ldots) \]
where
	\[ D_n(f) := \<0, 0, \ldots 0, \pi_0, \pi_1, \ldots \pi_n\>D^n(f), \]
the only thing we need to check here is that ${\cal D}(f)$ is a valid map in $\faa(\X)$.  That is, we need $\rs{D_n(f)} = \rs{\pi_n f}$.  But since $\rs{D(f)} = \rs{\pi_1f}$, we have $\rs{D^n(f)} = \rs{\pi_1\pi_1 \ldots \pi_1 f}$ ($n \ \pi_1$'s), so that 
	\[ \rs{D_n(f)} = \rs{\<0, 0, \ldots 0, \pi_0, \pi_1, \ldots \pi_n\>D^n(f)} = \rs{\pi_n f}, \]
as required.  
	
\end{proof}

\begin{corollary}
If $\X$ is a Cartesian restriction category, then $\faa(\X)$ is a generalized differential restriction category, with
	\[ D(f) := [\delta(f)]_1. \]
\end{corollary}

Again, this is nothing more than stating that free coalgebras exist; but it highlights the fact that there are many, many instances of generalized differential restriction categories beyond the standard examples.

\end{document}